\documentclass[a4paper, 12pt]{amsart}
\usepackage[utf8]{inputenc}
\usepackage[english]{babel}
\usepackage{amsmath, amsthm, amssymb}
\usepackage{graphicx}
\usepackage[dvipsnames]{xcolor}
\usepackage{caption}
\usepackage{subcaption}
\usepackage{mathtools}
\usepackage{enumerate}
\usepackage{tikz}
\usepackage{tikz-cd}
\usepackage{enumitem}
\usepackage{hyperref}
\hypersetup{ 
	colorlinks = true,
	allcolors = red,
}
\usepackage[capitalize,noabbrev]{cleveref}
\usepackage{ytableau}
\usepackage{faktor}

\usepackage[a4paper]{geometry}
\usepackage{titlesec}
\titleformat{\section}
{\normalfont\Large\bfseries}{\thesection}{1em}{}[{\titlerule[0.8pt]}]

\newtheorem{theorem}{Theorem}[section]
\newtheorem{proposition}[theorem]{Proposition}
\newtheorem{lemma}[theorem]{Lemma}
\newtheorem{corollary}[theorem]{Corollary}

\theoremstyle{definition}
\newtheorem{definition}[theorem]{Definition}

\newtheorem{remark}[theorem]{Remark}

\newtheorem{question}[theorem]{Question}

\DeclareMathOperator{\GL}{GL}

\DeclareMathOperator{\Syl}{Syl}

\DeclareMathOperator{\Aut}{Aut}

\DeclareMathOperator{\Alt}{Alt}

\DeclareMathOperator{\dG}{d}
\DeclareMathOperator{\Cyc}{C}

\DeclareMathOperator{\ZG}{Z}

\newcommand{\Z}{\mathbb{Z}}

\newcommand{\F}{\mathbb{F}}

\newcommand{\msc}[1]{\href{https://zbmath.org/classification/?q=#1}{#1}}
\newcommand{\gen}[1]{\langle #1\rangle}

\numberwithin{equation}{section}

\makeatletter
\renewcommand\subsection{\@startsection{subsection}{2}%
	\z@{.5\linespacing\@plus.7\linespacing}{-.5em}%
	{\normalfont \bfseries}}
\makeatother

\makeatletter
\@namedef{subjclassname@2020}{%
	\textup{2020} Mathematics Subject Classification}
\makeatother

\begin{document}

	\title{On finite $d$-maximal groups}
	\author[A.~Lucchini]{Andrea Lucchini}
	\address[Andrea Lucchini]{Universit\`a di Padova, Dipartimento di Matematica \lq\lq Tullio Levi-Civita\rq\rq}
	\email{lucchini@math.unipd.it}
	\author[L.~Sabatini]{Luca Sabatini}
	\address[Luca Sabatini]{ Alfr\'ed R\'enyi Institute of Mathematics, Budapest}
	\email{sabatini@renyi.hu}
 \author[M.~Stanojkovski]{Mima Stanojkovski}
	\address[Mima Stanojkovski]{Universit\`a di Trento, Dipartimento di Matematica}
	\email{mima.stanojkovski@unitn.it}

 \makeatletter
\@namedef{subjclassname@2020}{
 \textup{2020} Mathematics Subject Classification}
\makeatother

	
	\date{\today}

	\begin{abstract}
Let $d$ be a positive integer. A finite group is called $d$-maximal if it can be generated by precisely $d$ elements, while its proper subgroups have smaller generating sets. For $d\in\{1,2\}$, the $d$-maximal groups have been classified up to isomorphism and only partial results have been proved for larger $d$. In this work, we prove that a $d$-maximal group is supersolvable and we give a characterization of $d$-maximality in terms of so-called \emph{maximal $(p,q)$-pairs}. Moreover, we classify the maximal $(p,q)$-pairs of small rank obtaining, as a consequence, the classification of the isomorphism classes of $3$-maximal finite groups.
\end{abstract}

	\maketitle

\section{Introduction}

Let $G$ be a finite group and let $\dG(G)$ denote its minimum number of generators. Let $d$ be a positive integer.

\begin{definition}
	A finite group $G$ is said to be \emph{$d$-maximal} if $\dG(G)=d$ and, for every proper subgroup $H$ of $G$, one has $\dG(H)<d$.
\end{definition}

The only $1$-maximal groups are the cyclic groups of prime order, while the $2$-maximal groups -- also called \emph{minimal non-cyclic} -- have been classified by Miller and Moreno \cite{MM03}: up to isomorphism, if $G$ is a minimal non-cyclic group, then $G$ is an elementary abelian $p$-group of rank 2, the quaternion group $Q_8$, or there are distinct primes $p$ and $q$ such that $G=P\rtimes Q$ is a semidirect product of a cyclic group $P$ of order $p$ with a cyclic $q$-group $Q$ and $Q/C_Q(P)$ has order $q.$ The structure of $d$-maximal $p$-groups has been investigated by Laffey \cite{Laf73}. Adapting an argument of J.\ G.\ Thompson, he proved that, if $p$ is an odd prime and $P$ is a $d$-maximal $p$-group, then $P$ has class at most 2 and the Frattini subgroup of $P$ has exponent $p$ and coincides with its derived subgroup; in particular $|P|\leq p^{2d-1}.$  The situation for $p = 2$ turned out to be much more intricate. In 1996, Minh \cite{Min96} constructed a 4-maximal 2-group of class 3 and order $2^8.$ Nowadays groups of order $2^8$ can be examined using a computer. There are 20241 groups $G$ of order $2^8$ with $\dG(G)=4,$ and only two of them are 4-maximal with nilpotency class 3. All the known $d$-maximal 2-groups are of class at most $3$ and the following question is open:

\begin{question}
Are there $d$-maximal $2$-groups of arbitrary large nilpotency class?  
\end{question}

Clearly a nilpotent $d$-maximal group must be a $p$-group. In this paper, we are interested in $d$-maximal groups in the more general situation where $G$ is not nilpotent. Using the classification of the finite non-abelian simple groups, we prove that a finite $d$-maximal group is solvable and its order is divisible by at most two different primes, as the next result shows.

 \begin{theorem}\label{main}
 Let $G$ be a non-nilpotent $d$-maximal group. Then there exist distinct primes $p$ and $q$ such that the derived subgroup $P$ of $G$ is a Sylow $p$-subgroup of $G$ and $G/P$ is a cyclic $q$-group. Moreover, if $Q$ is a Sylow $q$-subgroup of $G$, then $Q/C_Q(P)~$ has order $q$.
 	\end{theorem}
  
In light of the previous result, it is natural to investigate the structure of finite $p$-groups that can occur as the derived subgroup of a non-nilpotent $d$-maximal group. 
We recall that a power automorphism of a finite group is an automorphism sending every subgroup to itself. If $G$ is an elementary abelian $p$-group, then a power automorphism  of $G$ is just scalar multiplication by some element of $\F_p^\times$.

\begin{definition} \label{defPair}
Let $p$ and $q$ be prime numbers. A \emph{maximal $(p,q)$-pair of rank $d$} is a pair $(P,\alpha)$ where
	$P$ is a finite $p$-group, $\alpha \in \Aut(P)$ has prime order $q$ dividing $p-1$, and the following properties are satisfied:
	\begin{enumerate}[label=$(\alph*)$]
		\item the minimal number of generators of every subgroup of $P$ is at most $\dG(P)=d$; \label{it:a}
		\item the image of $\alpha$ in $\Aut(P/\Phi(P))$ is a non-trivial power automorphism; \label{it:b}
		\item if $H$ is a proper subgroup of $P$ with $\dG(H)=\dG(P)$, then either $\alpha(H)\neq H$  or the image of $\alpha$ in $\Aut(H/\Phi(H))$ is not a non-trivial power automorphism. \label{it:c}
	\end{enumerate}
\end{definition}

We reformulate \cref{main} in terms of maximal pairs.

\begin{theorem} \label{thChar}
	A finite group $G$ is $d$-maximal if and only if one of the following occurs:
	\begin{enumerate}[label=$(\arabic*)$]
		\item the group $G$ is a $d$-maximal $p$-group;
		\item there exist a maximal $(p,q)$-pair $(P,\alpha)$ of rank $d-1$ and a cyclic $q$-group $\langle \beta \rangle$ 
		such that $G$ is isomorphic to $P\rtimes \langle \beta \rangle$ and, for every $x\in P$, one has $\alpha(x)=\beta(x)$. 
	\end{enumerate}
\end{theorem}

The Miller and Moreno classification of minimal non-cyclic groups can be essentially reformulated to saying that, if $(P,\alpha)$ is a maximal $(p,q)$-pair of rank $1,$ then $P$ has order $p.$ In \cref{sec5} we classify the maximal $(p,q)$-pairs $(P,\alpha)$ of rank 2, proving in particular that either $P$ has exponent $p$ and order at most $p^3$ or $(p,q)=(3,2),$ in which case there is a unique exceptional example with $P$ of order 81 and class 3. This result, combined with a recent classification of the $3$-maximal $p$-groups \cite{AZG22}, allows us to give in \cref{full3max} the full classification of the finite $3$-maximal groups. In \cref{sec6} we classify the maximal $(p,q)$-pairs $(P,\alpha)$ of rank 3: in this case $P$ has class at most 3 and order at most $p^6$, and if $|P|=p^6$, then $(p,q)=(3,2)$.

The behaviour of maximal pairs of small rank suggests the  following question.

\begin{question} Does there exist a function $f:\mathbb N\to \mathbb N$ with the property that
$|P|\leq p^{f(d)},$ whenever $(P,\alpha)$ is a maximal $(p,q)$-pair of rank $d?$
\end{question}
It follows from \cref{lem:basic1,lem:corba} that, if $(P,\alpha)$ is a maximal $(p,q)$-pair of rank $d$ and $P$ has class at most $c$, then $|P|\leq p^{cd}$; the previous question is thus equivalent to the following one:

\begin{question}
Does there exist a function $g:\mathbb N\to \mathbb N$ with the property that
$P$ has class at most $g(d),$ whenever $(P,\alpha)$ is a maximal pair of rank $d?$
\end{question}

We are not aware of examples of maximal pairs $(P,\alpha)$ with $P$ of class greater than $3$. This motivates the next problem:

\begin{question}
Is it possible to construct maximal pairs $(P,\alpha)$ with $P$ of arbitrarily large nilpotency class? 
\end{question}

The following is our main contribution to the solution of the previous questions. It implies, in particular, that the derived length of a $d$-maximal group of odd order is at most $3$.
 
	\begin{theorem}\label{th:q=2}
 Let $(P,\alpha)$ be a maximal $(p,q)$-pair.
 If $q>2$, then $P$ has class at most $2$.
	\end{theorem}

The proof of \cref{th:q=2} is given in \cref{sec:reg}, and involves results on maximal pairs $(P,\alpha)$ where $P$ is regular (the definition of regularity is given in Section \ref{subsecRegularity}). The class of regular $p$-groups is not only easier to study, but also a reasonable family to restrict to. Indeed, as soon as $p\geq 2d$ and $(P,\alpha)$ is a maximal pair of rank $d$, the group $P$ is regular (see \cref{lem:reg}). \\

\noindent
\textbf{Notation.}
We use standard group theory notation and write
\begin{itemize}
    \item $\ZG(G)$ for the center of $G$,
    \item $\Phi(G)$ for the Frattini subgroup of $G$,
    \item $(\gamma_i(G))_{i\geq 1}$ for the lower central series of $G$.
\end{itemize}
If $p$ is a prime number, $n$ a non-negative  integer, and $P$ a finite $p$-group, we write $\Omega_n(P)$ and $\mho_n(P)$ for the following subgroups:
\[
\Omega_n(P)=\gen{x\in P\mid x^{p^n}=1} \ \ \textup{ and } \ \ \mho_n(G)=\gen{x^{p^n} \mid x\in G}.
\]

\smallskip
\noindent
\textbf{Acknowledgements.} 
The second author has received funding from the European Research Council (ERC) under the European Union's Horizon 2020 research and innovation program (grant agreement No 741420).
The third author was funded by the Italian program Rita Levi Montalcini for young researchers, Edition 2020. We are thankful to the anonymous referees for their comments, which helped improve the exposition of this paper.

\section{Maximal groups and maximal pairs}

 In this section we translate the problem of classifying $d$-maximal groups into that of classifying maximal pairs, as defined in the Introduction.
	
\subsection{The structure of $d$-maximal groups}

The following theorem is \cite[Cor.~4]{Luc97}.
Its proof uses several different properties of the finite simple groups and requires their classification.
	
	\begin{theorem}\label{thLuc}
	Let $G$ be a finite group. Let $D = \max_{S \in \Syl(G)} \dG(S)$,
 where $S$ runs among the Sylow subgroups of $G$.
	 Then $\dG(G) \leq D+1$.
	 If $\dG(G)=D+1$, then there exists an odd prime $p$ and a quotient of $G$ isomorphic to
	 a semidirect product of an elementary abelian $p$-group $P$ of rank $D$ with a cyclic group $\langle \alpha \rangle$,
	 where $\alpha$ acts on $P$ as a non-trivial power automorphism.
	\end{theorem}

The next result describes the $d$-maximal groups with trivial Frattini subgroup. Note that, in the second case of \cref{propFF}, the subgroup $A$ will necessarily act on the elementary abelian group $P$ by scalar multiplication  by elements of $\F_p^\times$ and therefore its order $q$ will have to divide $p-1$, yielding, in particular, that $p$ is odd.
	
	\begin{proposition} \label{propFF}
	Let $G$ be a $d$-maximal finite group such that $\Phi(G)=1$. Then there exists a prime number $p$ such that one of the following holds:
	\begin{enumerate}[label=$(\arabic*)$]
	\item The group $G$ is an elementary abelian $p$-group of rank $d$;
	\item The group $G$ is isomorphic to a semidirect product $P\rtimes A$, where $P$ is an elementary abelian $p$-group of rank $d-1$ and $A$ is a central prime-order subgroup of $\GL_{d-1}(\F_p)$. 
	\end{enumerate}
	\end{proposition}
	\begin{proof} 
	If $G$ is nilpotent, then $G$ is a direct product of elementary abelian groups and (1) follows easily from $d$-maximality.
	Let $D = \max_{S \in \Syl(G)} \dG(S)$ and observe that $D < d$.
	From Theorem \ref{thLuc}, we obtain a normal subgroup $N$ of $G$ such that $G/N$ is isomorphic to a semidirect product $P\rtimes \gen{\alpha}$ where $P$ is elementary abelian of rank $d-1$ and $\alpha$ acts on $P$ as a non-trivial power automorphism. In particular, $p \neq 2$ and $\dG(G/N)=d$.
	 We claim that $N=1$. If this were not the case, since $\Phi(G)=1$, there would exist a maximal subgroup $M$ of $G$ such that $G=MN$ and thus
	$$ \dG(M) \geq \dG (M/(M\cap N)) = \dG(G/N) = d , $$
	which is impossible.
    Let $\sigma \in \langle \alpha \rangle$.
	If $\sigma$ acts non-trivially on $P$,
	then $\dG(P \rtimes \langle \sigma \rangle)=d$,
	which gives $\langle \sigma \rangle=\langle \alpha \rangle$.
	Since some Sylow subgroup of $\langle \alpha \rangle$ must act non-trivially on $P$, we conclude that $\alpha$ has prime-power order, say $q^t$.
	Moreover, $\alpha^q$ must act trivially on $P$.
	Since $\alpha^q \in \Phi(\langle \alpha \rangle)$, we conclude that $\alpha^q \in \Phi(G)=1$. The proof is complete.
	\end{proof}

	The following two results are well known. 
	The first is a direct consequence of the Schur-Zassenhaus Theorem (see \cite[9.3.5]{Robinson}), while the second is \cite[Thm.~1.6.2]{Khukhro}.
	
	\begin{lemma} \label{lemFrat}
	Let $G$ be a finite group. If $p$ divides $|G|$, then $p$ divides $|G:\Phi(G)|$.
	\end{lemma}

    \begin{lemma} \label{lemKhukhro}
    Let $G$ be a finite group, $\varphi$ an automorphism of $G$, and $N$ a normal $\varphi$-invariant subgroup whose order is coprime to the order of $\varphi$. Then $\mathrm{C}_{G/N}(\varphi) = \mathrm{C}_G(\varphi)N/N$. 
    \end{lemma}
	
	\begin{proposition} \label{propGen}
	Let $G$ be a $d$-maximal finite group and assume that $G$ is not a $p$-group. Then $G$ is isomorphic to a semidirect product $ P \rtimes \langle \alpha \rangle$, where $P$ is a $p$-group for some odd prime $p$,
	and $\alpha \in \Aut(P)$ has prime-power order $q^t$ for some $q$ dividing $p-1$.
	Moreover, $\dG(P)=d-1$, and $\alpha^q$ centralizes $P$. 
	\end{proposition}
	\begin{proof}
	Let $G$ be a non-nilpotent $d$-maximal group.
	Since $d=\dG(G)=\dG(G/\Phi(G))$, we have that $\overline{G}=G/\Phi(G)$ is $d$-maximal and Frattini-free. \cref{propFF} ensures the existence of an elementary abelian $p$-group $\overline{P}$ and $\overline{\alpha}\in\Aut(\overline{P})$ of prime order $q$ dividing $p-1$ such that $\overline{G}$ is isomorphic to the semidirect product $\overline{P}\rtimes\gen{\overline{\alpha}}$. As a consequence of Lemma \ref{lemFrat}, there exist integers $n \geq d-1$ and $t \geq 1$ such that $|G|=p^n q^t$.
	Since the Sylow $p$-subgroup $\overline{P}$ of $\overline{G}$ is normal, so is the Sylow $p$-subgroup $P$ of $G$.
	Therefore, $G$ can be written as a semidirect product $P \rtimes Q$, where $Q$ is a Sylow $q$-subgroup.
	Since $Q/\Phi(Q)$ is isomorphic to $(G/P)/\Phi(G/P)$, it follows from $(\Phi(G)P)/P \subseteq \Phi(G/P)$ that $Q/\Phi(Q)$ is a quotient of the cyclic group $G/(\Phi(G)P)$.
	So $Q=\langle \alpha \rangle$ for some $\alpha \in Q$, and $\dG(P)=d-1$.
	By $d$-maximality, $\alpha^q$ induces the identity on $P/\Phi(P)$. From Lemma \ref{lemKhukhro} we conclude that $\alpha^q$ must act trivially on the whole of $P$.
	\end{proof}

 \begin{remark}
 Let $D= \max_{S \in \Syl(G)} \dG(S)$.
 The inequality $d(G)>D$ plays a crucial role in our proof that a non-nilpotent $d$-maximal finite group $G$ must be solvable. However this inequality alone is not sufficient to deduce the solvability. Consider for example the direct product  $\Alt(5) \times H$, where $H$ is the semidirect product $(\F_{29})^2\rtimes \gen{\alpha}$ with $\alpha$ of order $7$ in $\F_{29}^\times$.
\end{remark} 

\subsection{Maximal $(p,q)$-pairs}\label{sec:pairs} 

Let $G$ be a non-nilpotent $d$-maximal group and let $P$ and $\alpha$ be as in \cref{propGen}. In particular, $\alpha^q$ generates a central subgroup of $G$ contained in $\Phi(G)$. It follows that the quotient $G/\langle \alpha^q \rangle$ is again $d$-maximal and of order $p^n q$, for some positive integer $n$. 
\cref{thChar} states that the study of these quotients is essentially equivalent to the investigation of maximal pairs.

\begin{proof}[Proof of Theorem \ref{thChar}]
It follows from Proposition \ref{propGen} that a $d$-maximal group $G$ satisfies either (1) or (2). Conversely, assume $G=P \rtimes \langle \beta \rangle$ is as described in (2) with $|\beta|=q^t$. This implies that $\dG(G)=d$. Let now $H$ be a proper subgroup of $G$. If $H$ is contained in $P$, then $\dG(H)<d$ by property \ref{it:a} of maximal pairs of rank $d-1$. So assume that $H$ is not contained in $P$, and let $Q$ be a Sylow $q$-subgroup of $H.$ If $|Q|< q^t,$ then $H$ is the direct product of $(H\cap P)$ and $Q$ and therefore $\dG(H)=\max(\dG(H\cap P),\dG(Q)) < d.$  Finally assume $|Q|=q^t$. Then, by the Sylow theorems, there exists $g\in G$ with $Q^g=\langle \beta \rangle$ and $H^g=(H^g \cap P)\langle \beta \rangle$. In particular $H^g\cap P$ is $\beta$-invariant, and property \ref{it:c} of maximal pairs of rank $d-1$ gives that $\dG(H)=\dG(H^g)< d$.
\end{proof}

\subsection{Actions through characters}\label{sec:basic}

In this section, let $A$ be a finite group and let $p$ be an odd prime. Let $\chi:A\rightarrow\Z_p^\times$ be a character. We define actions \emph{through characters} and present some related results that we will apply in the study of maximal $(p,q)$-pairs.

\begin{definition}
 The group $A$ is said to act on a group $G$ \emph{through $\chi$} if, for each $a\in A$ and $g\in G$, one has $g^a=g^{\chi(a)}$. 
\end{definition}

\begin{remark}\label{rmk:pair-char}
    Let $(P,\alpha)$ be a maximal $(p,q)$-pair and $A=\gen{\alpha}$.     Then, as a consequence of property \ref{it:b} of maximal pairs, there exists a non-trivial character $\chi:A\rightarrow\Z_p^\times$ such that $A$ acts on $P/\Phi(P)$ through $\chi$. 
Moreover, it follows from property \ref{it:c} that, if $H$ is a proper $\alpha$-invariant subgroup of $P$ with $\dG(H)=\dG(P)$ and such that $A$ acts on $H/\Phi(H)$ through a character $\chi_H$, then necessarily $\chi_H=1$. 
\end{remark}

The following result is straightforward.

\begin{lemma}\label{remDP}\label{remQuot}
    Let $(P,\alpha)$ and $(Q,\beta)$ be maximal $(p,q)$-pairs of ranks $d$ and $e$, respectively. Then the following hold:
    \begin{enumerate}[label=$(\arabic*)$]
        \item \label{it:quot} If $N$ is an $\alpha$-invariant normal subgroup of $P$ contained in $\Phi(P)$ and $\overline{\alpha}\in\Aut(P/N)$ is induced by $\alpha$, then $(P/N,\overline{\alpha})$ is also a  maximal $(p,q)$-pair of rank $d.$
        \item If $\gen{\alpha}$ and $\gen{\beta}$ act on $P/\Phi(P)$ and $Q/\Phi(Q)$ through the same character, then $(P\times Q, (\alpha,\beta))$ is a maximal $(p,q)$-pair of rank $d+e$.
    \end{enumerate}
\end{lemma}

The following results are taken from \cite[Sec.~2]{Sta21} and use the same notation.
In order, they are \cite[Lemma~2.5]{Sta21}, \cite[Lem.~2.6]{Sta21}, \cite[Cor.\ 2.12]{Sta21}, and \cite[Cor.\ 2.13]{Sta21}.

\begin{lemma}\label{lem:lcs-char}
    Let $P$ be a finite $p$-group that is also an $A$-group and assume that the induced action of $A$ on $P/\gamma_2(P)$ is through $\chi$. Then, for all integers $i\geq 1$, the induced action of $A$ on $\gamma_i(P)/\gamma_{i+1}(P)$ is through $\chi^i$.
\end{lemma}

\begin{lemma}\label{lem:surjective}
Let $P_1$ and $P_2$ be finite $p$-groups that are also $A$-groups, and assume that $A$ acts on $P_1$ through $\chi$. Moreover, let $\phi:P_1 \rightarrow P_2$ be a surjective homomorphism respecting the action of $A$, i.e.\ for all $a\in A$ and $g\in P_1$, one has that $\phi(g^a) = \phi(g)^a$. Then $A$ acts on $P_2$ through $\chi$. 
\end{lemma}

\begin{lemma}\label{lem:abelian+-}
  Let $P$ be a finite abelian $p$-group on which $A$ acts through $\chi$. Assume that $A=\gen{\alpha}$ has order $2$ and write 
  \[
  P^+=\{x\in P \mid \alpha(x)=x\} \ \ \textup{ and }\ \ P^-=\{x\in P \mid \alpha(x)=x^{-1}\}.
  \]
  Then $P=P^+\oplus P^-$.
\end{lemma}

\begin{lemma}\label{-on quotients}
Let $P$ a finite $p$-group on which $A$ acts through $\chi$.  
Let $N$ be a normal $A$-invariant subgroup of $P$ such that the restriction
of $\alpha$ to $N$ equals the inversion map $x\mapsto x^{-1}$. Assume, moreover, that also the automorphism of $P/N$ that is induced by $\alpha$ is equal to the inversion map. Then $\alpha$ is the inversion map on $P$ and $P$ is abelian.
\end{lemma}

\section{General results on maximal pairs}\label{sec:general}

Until the end of \cref{sec:general}, let $(P,\alpha)$ denote a maximal $(p,q)$-pair of rank $d$ and $A=\gen{\alpha}$. Moreover let  $\chi:A\rightarrow\Z_p^\times$ be the character through which $A$ acts on $P/\Phi(P)$ as in \cref{rmk:pair-char}.

\begin{lemma}\label{lem:basic1}
The following hold:
\begin{enumerate}[label=$(\arabic*)$]
    \item \label{it:easy1} one has $\Phi(P)=\gamma_2(P)$;
    \item \label{it:easy2} for each $i\geq 1$, the quotient $\gamma_i(P)/\gamma_{i+1}(P)$ is elementary abelian;
    \item \label{it:easy3} the induced action of $A$ on $\gamma_i(P)/\gamma_{i+1}(P)$ is through $\chi^i$.
\end{enumerate}
\end{lemma}

\begin{proof}
We start by proving \ref{it:easy1}. Applying \cref{remQuot}\ref{it:quot} to $N=\gamma_2(P)$, we assume without loss of generality that $P$ is abelian of exponent dividing $p^2$. Then $p$-th powering is a homomorphism $P\rightarrow\mho_1(P)$ and so it follows from \cref{lem:surjective} that $A$ acts on $\mho_1(P)$  through $\chi$. In order not to contradict property \ref{it:c} of maximal pairs, the group $P$ has to be equal to $\Omega_1(P)$, i.e.\ $P$ has exponent $p$. 
    Now \ref{it:easy2} immediately follows from \ref{it:easy1}, while \ref{it:easy3} is the combination of \ref{it:easy1} with \cref{lem:lcs-char}. 
\end{proof}

The following result follows directly from \cref{lem:basic1} and \cref{rmk:pair-char}.

\begin{lemma}\label{lem:corba}
Let $(P,\alpha)$ be a maximal pair of rank $d$ and let $i$ be a positive integer. Then the following hold:
\begin{enumerate}[label=$(\arabic*)$]
    \item one has $\dG(\gamma_i(P)/\gamma_{i+1}(P))\leq d$; 
    \item if $i\geq 2$ and $\dG(\gamma_i(P)/\gamma_{i+1}(P))=d$, then $\chi^i=1$.
\end{enumerate}
\end{lemma}

The following definition is taken from \cite[Sec.\ 2.3]{Sta21}. 

\begin{definition}
    Let $G$ be a finite $p$-group and let $H$ be a subgroup of $G$. 
A positive integer $j$ is called a \emph{jump of $H$ in $G$} if $H\cap\gamma_j(G)\neq H\cap\gamma_{j+1}(G)$.
\end{definition}

\begin{lemma}\label{lem:q-mod}
    Let $H$ be an $A$-invariant subgroup of $P$ for which the $p$-th powering map is an endomorphism. Then, for each jump $\ell$ of $\mho_1(H)$ there exists a jump $i$ of $H$ such that $i<\ell$ and $i\equiv \ell \bmod q$. 
\end{lemma}

\begin{proof}
    Let $\ell$ be a jump of $\mho_1(H)$ and let $y\in \mho_1(H)\setminus\{1\}$ be such that $y\in\gamma_{\ell}(P)\setminus \gamma_{\ell+1}(P)$. 
    Since $p$-th powering is an endomorphism of $H$, the subgroup $\mho_1(H)$ equals the set of $p$-th powers of elements of $H$.
    Let $x\in H$ be such that $x^p=y$  and let $i$ be the unique positive integer such that $x\in\gamma_i(P)\setminus\gamma_{i+1}(P)$. Then $i$ is a jump of $H$ and $i<\ell$ thanks to \cref{lem:basic1}\ref{it:easy2}. Moreover, the $p$-th powering map induces a surjective homomorphism $\gen{x}\gamma_{i+1}(P)/\gamma_{i+1}(P)\rightarrow \gen{y}\gamma_{\ell+1}(P)/\gamma_{\ell+1}(P)$. It follows from \cref{lem:surjective} and \cref{lem:basic1}\ref{it:easy3} that the induced action of $A$ on $\gen{y}\gamma_{\ell+1}(P)/\gamma_{\ell+1}(P)$ is both through $\chi^i$ and $\chi^\ell$. As the order of $\alpha$ is $q$, this implies that $i\equiv \ell \bmod q$. 
\end{proof}

\begin{corollary}\label{cor:i+q}
Let $i$ be a positive integer. Then $\mho_1(\gamma_i(P))$ is contained in $\gamma_{i+q}(P)\gamma_{2i}(P)$. 
\end{corollary}

\begin{proof}
    Write $\overline{P}=P/\gamma_{2i}(P)$ and use the bar notation for the subgroups of $\overline{P}$. Then $\gamma_i(\overline{P})=\overline{\gamma_i(P)}$ is abelian and therefore $p$-th powering on $\gamma_i(\overline{P})$ is an endomorphism. 
    It follows from \cref{lem:q-mod} that $\mho_1(\gamma_i(\overline{P}))$ is contained in $\gamma_{i+q}(\overline{P})$ and so we derive that $\gamma_i(P)$ is contained in $\gamma_{i+q}(P)\gamma_{2i}(P)$.
\end{proof}

\begin{corollary}\label{gamma4}
The group $\mho_1(\gamma_2(P))$ is contained in $\gamma_4(P)$.
\end{corollary}

\begin{proposition}\label{prop:no-exp-p}
 Assume that $P$ has class $3$. Then $P$ does not have exponent $p$.
\end{proposition}

\begin{proof}
For a contradiction, assume that $P$ has exponent $p$. If $M$ is a complement of $\gamma_2(P)\cap\ZG(P)$ in $\gamma_3(P)$, then \cref{remQuot} yields that $\overline{P}=P/M$ also belongs to a maximal pair and satisfies $\ZG(\overline{P})=\gamma_2(\overline{P})\cap\ZG(\overline{G})$.
We assume thus, without loss of generality, that $\gamma_2(P)\cap\ZG(P)=\gamma_3(P)$ and, additionally, that $|\gamma_3(P)|=p$. Write $|\gamma_2(P):\gamma_3(P)|=p^m$ and $C=\Cyc_P(\gamma_2(P))$.  From the non-degeneracy of the map $P/C\times \gamma_2(P)/\gamma_3(P)\rightarrow \gamma_3(P)$ we derive that $|P:C|=p^m$. It follows that 
    \[
|C|=\frac{|P|}{p^m}=\frac{|P:\gamma_2(P)|\cdot|\gamma_2(P):\gamma_3(P)|\cdot p}{p^m}=p^{d+1}.
    \]
Not to contradict property \ref{it:a} of maximal pairs the commutator subgroup of $C$ has to be nontrivial and so, being normal in $P$, we derive $\gamma_3(P)\subseteq \gamma_2(C)$. Note now that $\gamma_3(C)\subseteq [C,\gamma_2(P)]=1$ and so $C$ has class $2$. As the commutator map $C\times C\rightarrow \gamma_2(C)$ is bilinear, we conclude that there exist $x,y\in C\setminus\gamma_2(P)$ such that $[x,y]\in \gamma_3(C)$. This is a contradiction to $\chi\neq 1$. 
\end{proof}

\section{The structure of regular pairs}

In the wide world of $p$-groups, the subclass of regular groups is somewhat tamer, sharing, in some sense, a number of properties with abelian groups. In this section we study the effect of assuming regularity on a $p$-group $P$ that belongs to a maximal $(p,q)$-pair $(P,\alpha)$.
Moreover, we use regularity to prove general results on maximal pairs. 

\subsection{Regularity} \label{subsecRegularity}

Let $p$ be a prime number and let $P$ be a finite $p$-group. Then $P$ is said to be \emph{regular} if, for every $x,y\in P$, one has 
\[
(xy)^p\equiv x^py^p\bmod \mho_1(\gamma_2(\gen{x,y})). 
\]
The following lemma collects the properties of regular groups we will make use of. We refer the interested reader to 
\cite[Sec.~III.10]{Hup67} for more on regularity.

\begin{lemma}\label{lem:reg-basic}
    Let $p$ be a prime number and $P$ a finite $p$-group. Let, moreover, $\ell$ and $k$ be non-negative integers and $M$ and $N$ be normal subgroups of $P$. Then the following hold. 
    \begin{enumerate}[label=$(\arabic*)$]
        \item\label{it:reg1} If the class of $P$ is at most $p-1$, then $P$ is regular.
        \item\label{it:reg2} If the exponent of $P$ is $p$, then $P$ is regular.  
        \item\label{it:reg3} If the order of $P$ is smaller than $p^p$, then $P$ is regular.
         \item\label{it:reg7} If $|P:\mho_1(P)|<p^p$, then $P$ is regular.
        \item\label{it:reg4} If $P$ is regular, then $[\mho_\ell(M),\mho_k(N)]=\mho_{\ell+k}([M,N])$.
        \item\label{it:reg5} If $P$ is regular, then $\mho_k(P)=\{x^{p^k}\mid x\in P\}$ and $\Omega_{\ell}(P)=\{x\in P\mid x^{p^\ell}=1\}$.
        \item\label{it:reg6} If $P$ is regular, then $|\mho_k(P)|=|P:\Omega_k(P)|$. 
    \end{enumerate}
\end{lemma}

\begin{proof}
In order, these can be found in Satz~10.2(a)-(d), Satz~10.13, Satz~10.8(a), Satz~10.5, Satz~10.7(a), and Satz~10.13 from \cite[Ch.~III]{Hup67}.
\end{proof}

\begin{definition}
    A maximal $(p,q)$-pair $(P,\alpha)$ is called \emph{regular} if $P$ is regular.
\end{definition}

As \cref{lem:reg-basic} together with the following lemma show, regular pairs are very common among maximal $(p,q)$-pairs.

\begin{lemma}\label{lem:reg}
 Let $(P,\alpha)$ be a maximal $(p,q)$-pair of rank $d$. If $p\geq 2d,$ then $P$ is regular.
\end{lemma}

\begin{proof}
By \cref{prop:no-exp-p}, the quotient $P/\mho_1(P)$ has class at most $2$ and this implies that $|P:\mho_1(P)|\leq p^{2d-1}.$ Indeed, if the central $\gamma_2(P)$ had order $p^d$, we could easily construct an elementary abelian subgroup containing $\gamma_2(P)$ with index $p$, contradicting property \ref{it:a}. We derive that, if $p\geq 2d,$ then $|P:\mho_1(P)|\leq p^{p-1}$ and $P$ is regular by \cref{lem:reg-basic}\ref{it:reg7}.
\end{proof}

The next lemma is a stronger version of \cref{cor:i+q} for regular pairs. 

\begin{lemma}\label{lem:reg-i+q}
 Let $(P,\alpha)$ be a regular maximal $(p,q)$-pair and let $i>0$ be an integer. Then $\mho_1(\gamma_i(P))\subseteq \gamma_{i+q}(P)\gamma_{4i}(P)$.
\end{lemma}

\begin{proof}
    Thanks to the regularity assumption, the $p$-th powering map induces an endomorphism on $\gamma_i(P)/\mho_1(\gamma_{2i}(P))$. From \cref{lem:q-mod} and \cref{cor:i+q} we conclude that $\mho_1(\gamma_i(P)) \subseteq \gamma_{i+q}(P)\mho_1(\gamma_{2i}(P))\subseteq \gamma_{i+q}(P)\gamma_{4i}(P)$.
\end{proof}

\begin{proof}[Proof of Theorem \ref{th:q=2}]
We assume that $P$ has class at least $3$ and show that $q=2$. As a consequence of \cref{remQuot}, we assume without loss of generality that $P$ has class $3$. If $p=3$, we have that $q=2$ so we assume, additionally, that $p>3$. Then, by \cref{lem:reg-basic}\ref{it:reg1}, the group $P$ is regular.
Applying \cref{prop:no-exp-p} and \cref{lem:reg-i+q} with $i=1$, we 
obtain that $\{1\}\neq\mho_1(P) \subseteq \gamma_{q+1}(P)$. In particular, $q+1 \leq 3$ and so $q=2$.
\end{proof}

The derived length of odd order $d$-maximal groups is at most $3$.
The following restriction on their order follows.

\begin{proposition}
Let $G$ be a $d$-maximal group of odd order.
If $p$ is a prime and $G$ is a $p$-group, then $|G| \leq p^{2d-1}$.
Otherwise, there exist distinct primes $p$ and $q$ and integers $n\leq 2d-3$ and $t\geq 1$ such that $|G|=p^n q^t$.
\end{proposition}
\begin{proof}
If $G$ is a $p$-group, then the class of $G$ is at most $2$, and, $\gamma_2(G)$ being elementary abelian, $|G| \leq p^{2d-1}$ follows.
Otherwise, let $(P,\alpha)$ be as in \cref{thChar}.
From Theorem \ref{th:q=2}, we know that the class of $P$ is at most $2$. Now, the number $q$ being odd, the equality $|P|=|P:\gamma_2(P)||\gamma_2(P)|$ together with Lemma \ref{lem:corba} provides $n \leq d-1+d-2$, as desired.
\end{proof}

\subsection{Regular pairs}\label{sec:reg}

We now focus exclusively on regular pairs. Because of this,
until the end of this section, let $(P,\alpha)$ be a maximal regular $(p,q)$-pair of rank $d$. The results proven here are not only interesting for their own sake, but will be also applied in the study of maximal pairs of small rank.

\begin{lemma}\label{cor:class3P^p}
    Assume that $P$ has class $3$. Then $\mho_1(P)=\gamma_3(P)$. 
\end{lemma}

\begin{proof}
    Thanks to \cref{th:q=2} and \cref{lem:reg-i+q}, we know that $\mho_1(P)$ is contained in $\gamma_3(P)$ and, by \cref{prop:no-exp-p}, that $\mho_1(P)\neq 1$. If $\mho_1(P)$ were properly contained in $\gamma_3(P)$, modding out by $\mho_1(P)$ would contradict \cref{prop:no-exp-p}, so we conclude that $\gamma_3(P)=\mho_1(P)$.
\end{proof}

\begin{lemma}\label{lem:c-reg}
    Let $c\geq 3$ be the class of $P$. Then $\mho_1(\gamma_{c-2}(P))=\gamma_c(P)$. 
\end{lemma}

\begin{proof}
    We work by induction on $c$ and note that the case $c=3$ is given by \cref{cor:class3P^p}. Assume now that $c>3$ and that the result holds for $c-1$, in other words that $\gamma_{c-1}(P)=\mho_1(\gamma_{c-3}(P))\gamma_c(P)$. The subgroup $\gamma_c(P)$ being central, \cref{lem:reg-basic}\ref{it:reg4} yields the following:
    \begin{align*}
    \mho_1(\gamma_{c-2}(P))&=\mho_1([P,\gamma_{c-3}(P)])=[P,\mho_1(\gamma_{c-3}(P))] \\ &=[P,\mho_1(\gamma_{c-3}(P))\gamma_c(P)]=[P,\gamma_{c-1}(P)]=\gamma_c(P).
    \end{align*}
    This concludes the proof.
\end{proof}

\begin{proposition}\label{prop:reg-p}
    Assume the class of $P$ is at least $3$. Then $\mho_1(P)=\gamma_3(P)$.
\end{proposition}

\begin{proof}
    Let $c$ denote the class of $P$: we work by induction on $c$. The base of the induction is given by \cref{cor:class3P^p} so we assume that the result holds for $c-1$, i.e.\ that $\gamma_3(P)=\mho_1(P)\gamma_c(P)$.  We assume also, without loss of generality that $|\gamma_c(P)|=p$ and, for a contradiction, that $\gamma_c(P)$ is not contained in $\mho_1(P)$, i.e.\ that $\mho_1(P)\cap\gamma_c(P)=1$. It follows from \cref{th:q=2} and \cref{lem:reg-i+q} that $\mho_1(\gamma_{c-2}(P))=1$. 
    However, the subgroup $\gamma_c(P)$ being central,  \cref{lem:reg-basic}\ref{it:reg4} and \cref{lem:c-reg} yield
    \begin{align*}
    \{1\}=\mho_1(\gamma_{c-2}(P))& =\mho_1([P,\gamma_{c-3}(P)])=[P,\mho_1(\gamma_{c-3}(P))] \\ &=[P,\mho_1(\gamma_{c-3}(P))\gamma_c(P)]=[P,\gamma_{c-1}(P)]=\gamma_c(P).
    \end{align*}
    Contradiction.
\end{proof}

\begin{corollary}\label{cor:reg-i}
    Assume the class of $P$ is at least $3$ and let $i$ and $j$ be positive integers. Then the following hold: 
    \begin{enumerate}[label=$(\arabic*)$]
        \item\label{it:corr1} $\mho_1(\gamma_i(P))=\gamma_{i+2}(P)$.
        \item\label{it:corr2} If at least one of $i$ and $j$ is odd, then $[\gamma_i(P),\gamma_j(P)]=\gamma_{i+j}(P)$.
        \item\label{it:corr3} If $i=2k+1$, then $|\gamma_{i}(P):\gamma_{i+2}(P)|\leq p^d$. 
    \end{enumerate}
\end{corollary}

\begin{proof}
 \ref{it:corr1}   We work by induction on $i$ and note that the claim holds for $i=1$ thanks to \cref{prop:reg-p}. Assume now that $i>1$ and that $\mho_1(\gamma_{i-1}(P))=\gamma_{i+1}(P)$. It follows from  \cref{lem:reg-basic}\ref{it:reg4} that 
    \[
    \mho_1(\gamma_i(P))=\mho_1([P,\gamma_{i-1}(P)])=[P,\mho_1(\gamma_{i-1}(P))]=[P,\gamma_{i+1}(P)]=\gamma_{i+2}(P).
    \]
\ref{it:corr2} Without loss of generality assume that $i$ is odd and write $i=2k+1$. It follows from \cref{cor:reg-i} and \cref{lem:reg-basic}\ref{it:reg4} that 
    \[
    [\gamma_i(P),\gamma_j(P)]=[\mho_k(P),\gamma_j(P)]=\mho_k([P,\gamma_j(P)])=\mho_k(\gamma_{j+1}(P))=\gamma_{j+1+2k}(P)=\gamma_{i+j}(P).
    \]
\ref{it:corr3} Since $i>1$, Point \ref{it:corr1} yields that $\gamma_{i}(P)/\gamma_{i+2}(P)$ is elementary abelian. Not to contradict property \ref{it:a}, the number of generators of the last quotient is at most $d$.
\end{proof}

\begin{lemma}\label{lem:odd-stop}
Let $a$ be a positive integer and assume that  $|\gamma_{1+2a}(P):\gamma_{2+2a}(P)|= p$. Then the class of $P$ is $1+2a$.
\end{lemma}

\begin{proof}
As a consequence of \cref{cor:reg-i},  for any index $i$, one has $\mho_{a}(\gamma_i(P))=\gamma_{i+2a}(P)$. 
Moreover, $P$ being regular, we have $|\Omega_a(\gamma_i(P)|=|\gamma_i(P):\mho_a(\gamma_i(P))|=|\gamma_i(P):\gamma_{i+2a}(P)|.$
In particular, we derive
$$\frac{|P|}{|\gamma_2(P)\Omega_a(P)|}=\frac{|P|\cdot|\Omega_a(\gamma_2(P))|}{|\gamma_2(P)|\cdot|\Omega_a(P)|}=\frac{|P|\cdot|\gamma_2(P)|\cdot|\gamma_{1+2a}(P)|}{|\gamma_2(P)|\cdot|P|\cdot|\gamma_{2+2a}(P)|}=\frac{|\gamma_{1+2a}(P)|}{|\gamma_{2+2a}(P)|}=p.
$$
 It follows from \cref{cor:reg-i} and \cref{lem:reg-basic}\ref{it:reg4} that 
  \begin{align*}
	\gamma_{2+2a}(P)&=\mho_a(\gamma_2(P))=\mho_a([P,P])=\mho_a([P,\gamma_2(P)\Omega_a(P)])\\ 
	&=[P,\mho_a(\gamma_2(P))]=[P,\gamma_{2+2a}(P)]=\gamma_{3+2a}(P)
\end{align*}
and thus $\gamma_{2+2a}(P)=1$.
\end{proof}

\section{Maximal pairs of rank $2$}\label{sec5}

In this section, we classify the maximal $(p,q)$-pairs of rank $2$ and, as a consequence, the finite $3$-maximal groups. To this end, until the end of \cref{sec5}, let  $(P,\alpha)$ be a maximal $(p,q)$-pair of rank $2$. 

\begin{proposition}\label{prop:d2max}
The group $P$ has maximal class. 
\end{proposition}

\begin{proof}
 We fix $i\geq 2$ and show that $|\gamma_i(P):\gamma_{i+1}(P)|\leq p$. Since $d=2$, we know that $|\gamma_i(P):\gamma_{i+1}(P)|\leq p^2$. Assume for a contradiction that $|\gamma_i(P):\gamma_{i+1}(P)|= p^2$.  
Note that $\gamma_i(P)/\gamma_{i+2}(P)$ is abelian and, thanks to \cref{cor:i+q}, its exponent divides $p$. 
 Then, since $P/\gamma_2(P)$ is a $2$-dimensional vector space over $\F_p$ and the commutator map induces a surjective homomorphism $\wedge^2(P/\gamma_2(P))\rightarrow \gamma_2(P)/\gamma_3(P)$, we have that $i>2$. In particular, $\gamma_{i-1}(P)/\gamma_{i+1}(P)$ is abelian, of order at least $p^3$, and, by \cref{cor:i+q}, of exponent $p$. This gives a contradiction to property \ref{it:a} of maximal pairs.
\end{proof}

\begin{lemma}\label{prop:n5p+1}
  If $p>3$, then $P$ has order dividing $p^4$.
\end{lemma}

\begin{proof}
Write $|P|=p^n$ and assume, for a contradiction, that $n\geq 5$. Thanks to \cref{prop:d2max} the group $P$ has maximal class and thus a unique quotient $\overline{P}$ of order $p^5$ and class $4\leq p-1$. 
Thanks to \cref{lem:reg-basic}\ref{it:reg1}, the group $\overline{P}$ is regular and so \cref{lem:odd-stop} yields that $\overline{P}$ has class $3$. Contradiction. 
\end{proof}

\begin{proposition} \label{prop:d=2p>3}
   Assume that $p>3$.  Then $P$ is is isomorphic to one of the following:
\begin{enumerate}[label=$(\arabic*)$]
    \item an elementary abelian group of order $p^2$;
    \item an extraspecial group of order $p^3$ and exponent $p$.
\end{enumerate}
\end{proposition}

\begin{proof}
We first prove that $|P| \leq p^3$.
For a contradiction, suppose that $|P| \geq p^4$. Then \cref{prop:n5p+1} yields that $|P|=p^4$ and, by \cref{prop:d2max}, the class of $P$ is $3$. 
It follows from \cref{lem:reg-basic}\ref{it:reg1} that $P$ is regular and from \cref{prop:reg-p} that $\mho_1(P)=\gamma_3(P)$. 
Moreover, \cref{th:q=2} ensures that $q=2$. It is easily seen that $C=\Cyc_P(\gamma_2(P))$ is abelian of order $p^3$. The rank of $P$ being $2$, this implies that $\mho_1(C)=\gamma_3(P)$ and so $C$ is different from $M=\Omega_1(P)$, which is also a maximal subgroup of $P$ (see \cref{lem:reg-basic}\ref{it:reg6}). Since both $C$ and $M$ contain $\Phi(P)$, both subgroups are $A$-invariant. 
Write now $\overline{M}=M/\gamma_3(P)$ and note that $\overline{M}$ is abelian and $A$-invariant. Then \cref{lem:abelian+-} implies that $\overline{M}=\overline{M}^+\oplus\overline{M}^-$ where both summands have order $p$. Let $N$ be the unique subgroup of $M$ mapping to $\overline{M}^-$ in $\overline{M}$.  Since $A$ acts on $\gamma_3(P)$ through $\chi^3=\chi$, we derive from \cref{-on quotients} that $N$ is an elementary abelian subgroup of order $p^2$ on which $A$ acts through $\chi$. This gives a contradiction to property \ref{it:c} of maximal pairs and $d=2$.

We have proved that $|P|\leq p^3$ and so $|P|$ is $p^2$ or $p^3$.
If $|P|=p^2$, then clearly $P$ is elementary abelian; assume therefore that $|P|=p^3$. Thanks to \cref{lem:reg-i+q} the exponent of $P$ is equal to $p$ and, the rank of $P$ being $2$, the group $P$ is non-abelian.
\end{proof}

\begin{proposition}\label{prop:d=2p=3}
Assume that $p=3$.  Then $P$ is is isomorphic to one of the following:
\begin{enumerate}[label=$(\arabic*)$]
    \item an elementary abelian group of order $9$;
    \item an extraspecial group of order $27$ and exponent $3$;
    \item the group \emph{\texttt{SmallGroup(81,10)}}.
\end{enumerate}
\end{proposition}

\begin{proof}
The claim is easily verified when $|P| \leq 27$, we assume therefore that $|P|\geq 81$. 
 The remaining part of the proof is computational and has been checked by all three authors in the computer algebra systems \texttt{GAP} \cite{GAP4} and \texttt{SageMath} \cite{SageMath}.

 Thanks to \cref{prop:d2max}, we know that $P$ has maximal class.  There exist precisely $4$ groups of order $3^4=81$ and maximal class up to isomorphism: these are the groups 
 \texttt{SmallGroup(81,7)}, \texttt{SmallGroup(81,8)}, \texttt{SmallGroup(81,9)}, \texttt{SmallGroup(81,10)} in the  \texttt{SmallGroup} library of \texttt{GAP} \cite{BEO02}. Each of these groups has an automorphism $\alpha$ of order $2$ that induces scalar multiplication by $-1$ on the Frattini quotient. For each of these groups other than \texttt{SmallGroup(81,10)}, the subgroup generated by the elements of order $3$ has order at least $27$: this ensures that the group has a subgroup of order $p^2$ on which $\alpha$ acts as scalar multiplication by $-1$, contradicting property \ref{it:c}. On the contrary, the subgroup of \texttt{SmallGroup(81,10)} that is generated by the elements of order $3$ is equal to the derived subgroup of \texttt{SmallGroup(81,10)}, from which it is not difficult to deduce that $(\texttt{SmallGroup(81,10)},\alpha)$ is a maximal pair of rank $2$ yielding the $3$-maximal group \texttt{SmallGroup(162,22)}.

If we now move to the groups of order $3^5=243$, we find that \texttt{SmallGroup(243,26)} is the unique $3$-group, up to isomorphism, of maximal class and order $243$ that possesses an automorphism $\beta$ of order $2$ that induces scalar multiplication by $-1$ on the Frattini quotient. However, the quotient of \texttt{SmallGroup(243,26)} by its center is isomorphic to \texttt{SmallGroup(81,9)} and thus not isomorphic to \texttt{SmallGroup(81,10)}. As a consequence of \cref{remQuot} we derive that $\texttt{SmallGroup(243,26)}$ is not part of any maximal pair and our classification is therefore complete.
\end{proof}

\subsection{The classification of $3$-maximal groups}\label{full3max}

Combining \cite[Thm.~1.11, Prop.~4.3]{AZG22} (used for (1) and (2)) with \cref{thChar},
and Propositions \ref{prop:d=2p>3} and \ref{prop:d=2p=3} (used for (3)), we obtain the list of $3$-maximal finite groups.
Specifically, a finite group $G$ is $3$-maximal if and only if one of the following occurs:
\begin{enumerate}
	\item There exists an odd prime $p$ such that $G$ is a $p$-group.
 Moreover $G$ is isomorphic to one of the following groups:
	\begin{enumerate}[label=(\roman*)]
		\item  an elementary abelian group of order $p^3$;
		\item the group of order $p^4$ defined by
		$$\langle
		a, b, c \mid a^{p^2}=b^p=c^p =[a,b]=[a, c]=1, [c, b] = a^p\rangle. $$
	\end{enumerate} 
 
	\item The group $G$ is a 2-group. More precisely $G$ is isomorphic to one of the following groups:
	\begin{enumerate}[label=(\roman*)]
		\item an elementary abelian group of order $8$;
		\item the direct product $C_2\times Q_8$;
		\item the central product $C_4*Q_8=C_4*D_8$;
		\item \texttt{SmallGroup(32,32)}. 
	\end{enumerate}
 
\item There exist an odd prime $p$ and a positive integer $t$ such that 
$G$ is a semidirect product $P\rtimes \langle \alpha\rangle$ where $P$ is a $p$-group, 
$\alpha$ has order $q^t$ for some prime $q$ that divides $p-1$, $\alpha^q\in\ZG(G)$, and $G/\langle \alpha^q\rangle$ is isomorphic to one of the following:
\begin{enumerate}[label=(\roman*)]
    \item a semidirect product $P\rtimes C_q$, where $P$ is elementary abelian of order $p^2$;
    \item a semidirect product $P\rtimes C_q$ with $P$ extraspecial of exponent $p$ and order $p^3$;
	\item \texttt{SmallGroup(162,22)}.

\end{enumerate}

\end{enumerate}

\section{Maximal pairs of rank $3$}\label{sec6}

In order to gather more evidence in the direction of answering the questions from the Introduction, in this section we completely classify the maximal $(p,q)$-pairs of rank $3$.

\begin{lemma}\label{lem:rk3cl2}\label{lem:rk3cltwo}
Let $(P,\alpha)$ be a maximal $(p,q)$-pair of rank $3$. 
If $P$ has nilpotency class $2$, then $P$ has order $p^4$, exponent $p$, and $\gamma_2(P)$ of order $p$.
\end{lemma}

\begin{proof}
For a contradiction, let $(P,\alpha)$ be a maximal $(p,q)$-pair with $\dG(P)=3$ and $\gamma_2(G)$ central of order $p^2$. The group $P$ is regular by \cref{lem:reg-basic}\ref{it:reg1}, and it has exponent $p$ thanks to \cref{lem:reg-i+q}. 
Define $V=P/\ZG(P)$ and $W=\gamma_2(P)$, with $\dim W=2$. Then the commutator map induces a surjective homomorphism 
$\phi:\wedge^2V\rightarrow W$, showing in particular that $V$ has dimension $3$ (otherwise $\dim\wedge^2V=1$). It follows that $\phi$ has a $1$-dimensional kernel, spanned by $g\ZG(P)\wedge h\ZG(P)$, say. Then the subgroup generated by $g$, $h$, and $\gamma_2(G)$ is an abelian group of order $p^4$ and exponent $p$. This gives a contradiction to property \ref{it:a}. We have proved that $|\gamma_2(p)|=p$ and thus $P$ has order $p^4$.
\end{proof}

\begin{lemma}\label{lemmap5}
Let $(P,\alpha)$ be a maximal $(p,q)$-pair of rank $3$. 
If $p>3$ and $|P|=p^5$, then the following hold: 
\begin{enumerate}[label=$(\arabic*)$]
	\item The group $\gamma_2(P)$ is isomorphic to $C_p\times C_p.$ 
	\item The group $\gamma_3(P)$ is isomorphic to $C_p.$ 
	\item The group $\Cyc_P(\gamma_2(P))$ is isomorphic to $ C_{p^2}\times C_p\times C_p$. 
	\item The order of $\Omega_1(P)$ is equal to $p^4.$
	\item One has $\gamma_2(\Omega_1(P))=\gamma_2(P).$
	\item One has  $q=2.$
\end{enumerate}
\end{lemma}
\begin{proof}
Assume that $|P|=p^5$ and that $p>3$. Then $P$ has class $3$ by \cref{lem:rk3cltwo} and it is regular by \cref{lem:reg-basic}\ref{it:reg1}. In particular, we have that $|\gamma_2(P)|=p^2$ and $|\gamma_3(P)|=p$ and, thanks to \cref{th:q=2}, that $q=2$. Moreover, $\gamma_2(P)$ has exponent $p$ by \cref{gamma4} and thus (1)-(2)-(6) are proved. 
Set now $C=\Cyc_P(\gamma_2(P))$ and note that $C$ is maximal in $P$. Then it holds that $[P,C]=\gamma_2(P)$ and also that $[C,[P,C]]=[C,\gamma_2(P)]=1$. By the Three Subgroups Lemma we have $[P,[C,C]]=1$ yielding that $[C,C]$ is contained in $\gamma_3(P)$. Then 
the commutator map induces a bilinear map 
$C/\gamma_2(G)\times C/\gamma_2(G)\rightarrow\gamma_3(G)$ 
yielding that either $\gen{\alpha}$ acts on $\gamma_3(P)$ through $\chi^2=1$ or $[C,C]=1$. Since $\chi\neq 1$, we derive  that $C$ is abelian. Not to contradict  property \ref{it:a} we have therefore that $\exp(C)\neq p$ and, as a consequence of \cref{prop:reg-p}, that $\mho_1(C)=\gamma_3(P)=\mho_1(P)$. This  and proves (3) and \cref{lem:reg-basic}\ref{it:reg6} takes care of (4). We conclude by proving (5). To this end, let $M=\Omega_1(P)$ and, for a contradiction, assume that $\gamma_2(M)\subsetneq\gamma_2(P)$. If $M$ is abelian, then we have a contradiction to property \ref{it:a}, so, since $\gamma_3(P)=\gamma_2(P)\cap\ZG(P)$ and $\gamma_2(M)$ is normal in $P$, it holds that $\gamma_2(M)=\gamma_3(P)$. Define now $\overline{P}=P/\gamma_3(P)$ and use the bar notation for the subgroups of $\overline{P}$. The automorphism $\alpha$ induces an automorphism $\overline \alpha$ of $\overline P$ and it follows from \cref{lem:abelian+-} that $\overline{M}=\overline{M}^+\oplus\overline{M}^-$. Let now $N$ be a subgroup of $P$ that contains $\gamma_3(P)$ and such that $\overline{N}=\overline{M}^-$. Since $\gamma_3(P)=\gamma_3(P)^-$, it follows from \cref{-on quotients} that $N=N^-$ and $N$ is abelian. Since $N$ is contained in $M$, this yields a contradiction to property \ref{it:c} of maximal pairs. \end{proof}

\begin{proposition}\label{prop:unique-p5}
Let $(P,\alpha)$ be a maximal $(p,q)$-pair of rank $3$. 
If $p>3$ and $|P|=p^5,$ then $q=2$ and $P$ is uniquely determined up to isomorphism. Indeed $P$ is isomorphic to
	$$\begin{aligned} X=\langle x_1,x_2,x_3,x_4,x_5\mid  x_1^p=x_5, x_2^p=x_3^p=x_4^p=x_5^p=1,[x_2,x_3]=x_4,[x_2,x_4]=x_5,\\ [x_1,x_2]=[x_1,x_3]=[x_1,x_4]=[x_1,x_5]=[x_2,x_5]=[x_3,x_4]=[x_3,x_5]=[x_4,x_5]=1\rangle\end{aligned}$$
 where the following hold:
 \begin{itemize}
     \item the group $Y=\gen{x_1, x_3, x_4}$ is a maximal abelian subgroup of $X$;
     \item $\gen{x_2}$ is a complement of $Y$ in $X$;
     \item one has $\gen{x_5}=\mho_1(X)=\ZG(X)$.
 \end{itemize}
\end{proposition}
\begin{proof}
The result could be deduced from the list of finite groups of order $p^5$, given by Bender in \cite{Ben27}. In that paper the groups of order $p^5$ are divided in 54 families, and the only one satisfying  the conditions obtained in the previous lemma is the unique group in family 23. We prefer to give a direct proof. Let $C=C_P(\gamma_2(P))$ and $M=\Omega_1(P).$ It follows from \cref{lemmap5}, that $$\gamma_2(P)\subseteq M\cap C\cong C_p\times C_p\times C_p,$$  so we may choose $x_2,x_3,x_4,x_5$ so that 
$$M=\langle x_2,x_3,x_4,x_5\rangle, \quad
M\cap C=\langle x_3,x_4,x_5\rangle,\quad \gamma_2(P)=\langle x_4,x_5\rangle, \quad 
\gamma_3(P)=\langle x_5\rangle.$$ Since $[M,M]=\gamma_2(P)$ and $[M,M]=[x_2, C\cap M],$ we may choose $x_3, x_4$ so that $[x_2,x_3]=x_4$ and $[x_2,x_4]=x_5.$ Now let $y\in C\setminus M.$ Since $\gamma_2(P)=[x_2,M\cap C],$ there exists $x\in M\cap C$ such that $[yx,x_2]=1.$ This implies $x_1=yx\in Z(P).$ Then $x_1^p=y^p\in P^p=\gamma_3(P),$ so it is not restrictive to assume $x_1^p=x_5.$
\end{proof}

\begin{remark}
	Let $P=X$ be the group described in \cref{prop:unique-p5}. Then the map that sends $x_4$ to $x_4$ and $x_i\to x_i^{-1}$ if $i\in \{1,2,3,5\}$ can be extended to an automorphism $\alpha$ of $P$ of order 2. We verify that $(P,\alpha)$ is a maximal $(p,2)$-pair of rank 3. Suppose that $H$ is a proper subgroup of $P$ with $\dG(H)\geq \dG(P)=3.$ Then $p^3\leq |H|\leq p^4.$ If
	$|H|=p^4,$ then $H$ is a maximal subgroup of $P$, and therefore $\Phi(P)=\langle x_4,x_5\rangle \subseteq H$. Moreover either $H=\Omega_1(P)$ or ${\rm{exp}}(H)=p^2.$ In any case $x_5$ belongs to $\Phi(H)$. So $\dG(H)\leq 3$ and if $\dG(H)=3,$ then $\Phi(H)=\langle x_5\rangle$. In the latter case, since $x_4\in H$ and $\alpha(x_4)=x_4,$ the map $\alpha$ does not induce a non-trivial power automorphism of $H/\Phi(H).$ Finally suppose that $H$ is elementary abelian of order $p^3$ and that $\alpha$ induces a non-trivial power automorphism on $H.$ It must be that $H$ is contained in $\Omega_1(P)$ and $x_4\notin H.$ This is impossible, since $H$ would be a maximal subgroup of $\Omega_1(P)$ and it would contain $\Phi(\Omega_1(P))=\langle x_4,x_5\rangle.$
\end{remark}

\begin{proposition}\label{prop:p5}
Let $(P,\alpha)$ be a maximal $(p,q)$-pair of rank $3$. 
If $p>3,$ then the order of $P$ is at most $p^5.$
\end{proposition}

\begin{proof}
 Assume for a contradiction that $|P|=p^6$ and that $p>3$. Then $P$ has class at least $3$ by \cref{lem:rk3cl2} and, since $P$ has rank $3$, the index $|\gamma_3(P):\gamma_4(P)|$ is either $p$ or $p^2$. The group $P$ is regular thanks to \cref{lem:reg-basic}\ref{it:reg1}. Since In the first case, \cref{lem:odd-stop} yields that $P$ has class $3$ and that $|\gamma_2(P):\gamma_3(P)|=p^2$: this contradicts \cref{lem:rk3cl2} combined with \cref{remQuot}. 
 We have thus proved that $P$ has class $3$ and that $|\gamma_3(P)|=p^2$. 
 Observe now that the surjective homomorphism $\wedge^2(P/\gamma_2(P))\rightarrow\gamma_2(P)/\gamma_3(P)$ that is induced by the commutator map has a non-trivial kernel. We fix $g\gamma_2(P)\wedge h\gamma_2(P)\neq 0$ in such kernel and define $M=\gen{g,h}\gamma_2(P)$. Then $M$ has order $p^5$ and $\gamma_2(M)$ is contained in $\gamma_3(P)$. Since $\chi\neq 1$, it follows that $\gamma_2(M)=1$, contradicting the fact that 
$|\Cyc_P(\gamma_2(P)):\gamma_2(P)|=p$. 
\end{proof}

\begin{remark}
The example described in \cref{prop:unique-p5} occurs also when $p=3.$ There exists, however, another  non-isomorphic maximal $(3,2)$-pair with $P$ of order $3^5,$ namely $P$ is the direct product $C_3\times X,$ where $X$ is isomorphic to \texttt{SmallGroup(81,10)}. This is indeed a consequence of \cref{remDP} and \cref{prop:d=2p=3}.
A computational check through the \texttt{SmallGroup} library of \texttt{GAP} reveals that there is also a unique possibility of order $3^6$: if $\tilde P$ is equal to \texttt{SmallGroup(729,148)},
then $\tilde P$ has an automorphism $\tilde \alpha$ of order 2, with the property that the semidirect product 
$\tilde P\rtimes \langle \tilde \alpha \rangle$, which is isomorphic to \texttt{SmallGroup(1458,805)}, is a 4-maximal group. This information allows us to prove that there exists no maximal
$(3,2)$-pair $(P,\alpha)$ of rank 3 with $|P|\geq 3^7.$ For this purpose it suffices to exclude the possibility $|P|=3^7.$ Assume by contradiction that such a group $P$ exists. Then $\tilde P$ would be an epimorphic image of $P.$ Since $\tilde P$ has nilpotency  class 3, the class of $\tilde P$ would be either 3 or 4.
In the first case,  $|\gamma_3(P)|=3\cdot | \gamma_3(\tilde{P})|=3^3,$ but this is impossible. So $|\gamma_4(P)|=3$ and $P/\gamma_4(P)\cong \tilde P.$ There are 1023 groups $P$ with $|P|=3^7$ satisfying $|\gamma_2(P)|=|\Phi(P)|=3^4$ and $|\gamma_4(P)|=3$, but a computational check shows that none of them satisfies $P/\gamma_3(P)\cong \tilde P.$
\end{remark}

\begin{remark}
When $d>3,$ there exist  maximal  pairs $(P,\alpha)$ of rank $d$, with $P$ of class 2, but $\gamma_2(P)$ is
non-cyclic. For example, there are three maximal pairs $(P,\alpha)$ or rank $4$, up to isomorphism, such that $|P|=3^6$, and $\gamma_2(P)\cong C_3\times C_3$. These are \texttt{SmallGroup(1458,1540)}, \texttt{SmallGroup(1458,1576)}, \texttt{SmallGroup(1458,1613)}. 
\end{remark}

\end{document}